\documentclass[a4paper,11pt]{amsart}
\usepackage[utf8]{inputenc}
\usepackage[english]{babel}
\usepackage{amsaddr}
\usepackage{calrsfs}
\usepackage{amssymb}
\usepackage{amsmath}
\usepackage{mathrsfs}
\usepackage{bm}
\usepackage[toc,page,header]{appendix}
\usepackage{graphicx,color}
\usepackage[dvipsnames]{xcolor}
\usepackage[colorlinks=true,linkcolor=blue, anchorcolor=blue, citecolor=blue, urlcolor=blue, filecolor=blue, menucolor=blue, pagecolor=blue]{hyperref}
\usepackage{tikz}
\usetikzlibrary{decorations.markings}
\usepackage{booktabs}
\usepackage{multirow}
\usepackage{float}

\makeatletter
\@namedef{subjclassname@2020}{%
	\textup{2020} Mathematics Subject Classification}
\makeatother

\numberwithin{equation}{section}

\theoremstyle{definition}
\newtheorem{dfn}{Definition}[section]

\theoremstyle{plain}
\newtheorem{thm}{Theorem}[section]
\newtheorem{pro}{Proposition}[section]

\newtheorem{lmm}{Lemma}[section]
\theoremstyle{remark}
\newtheorem{rem}{Remark}[section]

\newcommand{\N}{\mathbb{N}}

\newcommand{\E}{\mathbb{E}}
\renewcommand{\P}{\mathbb{P}}
\newcommand{\Var}{\mathbb{V}ar}
\newcommand{\Cov}{\mathrm{C}ov}

\newcommand{\R}{\mathbb{R}}

\newcommand{\e}{\mathrm{e}}

\newcommand{\lb}{\lbrace}
\newcommand{\rb}{\rbrace}

\renewcommand{\d}{\mathrm{d}}

\renewcommand{\i}{\mathrm{i}}

\begin{document}
\title[Multi-Mixed FBM and FOU]{Multi-Mixed Fractional Brownian Motions and Orstein--Uhlenbeck Processes}

\date{\today}

\author[Almani]{Hamidreza Maleki Almani}
\address{School of Technology and Innovations, University of Vaasa, P.O. Box 700, FIN-65101 Vaasa, FINLAND}
\email{hamidreza.maleki.almani@uwasa.fi}

\author[Sottinen]{Tommi Sottinen}
\address{School of Technology and Innovations, University of Vaasa, P.O. Box 700, FIN-65101 Vaasa, FINLAND}
\email{tommi.sottinen@uwasa.fi}


\begin{abstract}
We study the so-called multi-mixed fractional Brownian motions (mmfBm) and multi-mixed fractional Ornstein--Ulhenbeck (mmfOU) processes. These processes  are constructed by mixing by superimposing (infinitely many) independent fractional Brownian motions (fBm) and fractional Ornstein--Uhlenbeck processes (fOU), respectively. 
We prove their existence as $L^2$ processes and study their path properties, viz. long-range and short-range dependence, H\"older continuity, $p$-variation, and conditional full support. 
\end{abstract}



\maketitle


\section{Introduction and Preliminaries}\label{sect:introduction}

The fractional Brownian motion (fBm) $B^H$, with parameter $H\in(0,1)$ called the Hurst index, is the unique (up to a multiplicative constant) centered $H$-self-similar stationary-increment Gaussian process. The fBm was first studied in  \cite{Kolmogorov-1940}. The name fractional Brownian motion comes from the influential article \cite{Mandelbrot-Van-Ness-1968}. For further information of the fBm, see the monographs \cite{Biagini-Hu-Oksendal-Zhang-2008,Mishura-2008}. The covariance of the fBm with Hurst index $H$ is given by
$$
r_H(t,s) 
= \frac12\left[t^{2H} + s^{2H} - |t-s|^{2H}\right]. 
$$
For $H=1/2$ this process is well-known as the Brownian motion (Bm) or the Wiener process.
As a stationary-increment process, the fBm admits the spectral density
\begin{equation}\label{fbm-sd}
f_H(x)
=
\frac{\sin(\pi H)\Gamma(1+2H)}{2\pi}\, |x|^{1-2H},
\end{equation}
where $\Gamma$ is the complete gamma function
$$
\Gamma(\alpha) = \int_0^\infty t^{\alpha-1}\e^{-t}\, \d t, 
$$
see \cite{Samorodnitsky-Taqqu-1994}. 

Let
$$
\varrho_H(\delta;t)
=
\E\left[(B^H_{\delta}-B^H_0)(B^H_{t+\delta}-B^H_t)\right]
$$
be the incremental autocovariance (with lag $\delta$) of the fBm.
For $t\to\infty$ we have the power decay
$$
\varrho_H(\delta;t) 
\sim
H(2H-1)\delta^2 t^{2H-2}.
$$
This means that the increments of fBm, called the fractional Gaussian noise (fGn), are positively correlated and long-range dependent of $H>\frac12$. For $H<\frac12$ they are negatively correlated and short-range dependent.

In the Bm case $H=\frac12$ we have independent increments, i.e., no dependence:
$$
\varrho_{\frac12}(\delta;t) =0.
$$

A process $X$ is (locally) H\"older continuous with exponent $H$ if
$$
\sup_{t,s\in[0,T]}\frac{|X_t-X_s|}{|t-s|^H} < \infty
$$
The H\"older index of a process $X$ is
$$
\mathrm{Hol}_T(X) = \sup\left\{ H>0\,;\, \sup_{t,s\in[0,T]}\frac{|X_t-X_s|}{|t-s|^H} < \infty\right\}
$$

The fBm has almost surely locally H\"older continuous paths with any order $H-\varepsilon$ for any $\varepsilon>0$, but not with order $H$.  
This follows, e.g., from Theorem 1 of \cite{Azmoodeh-Sottinen-Viitasaari-Yazigi-2014}. 
Consequenlty, $\mathrm{Hol}(B^H) = H$.

In addition to H\"older continuity, we have the equidistant $p$-variation as a measure of the path regularity. For a process $X$ and $p\in [1,\infty)$ for the equidistant partitions
$
\pi_n := \lbrace t_k=\frac{k}{n}T: k=0, 1,\ldots , n\rbrace,
$ 
consider the limit in probability
$$
V^p_T(X) := \lim_{n\to\infty}\sum_{k=1}^n |X_{t_k}- X_{t_{k-1}}|^p. 
$$
If this limit is finite, it is called the equidistant $p$-variation on $[0,T]$ of $X$. The equidistant $p$-variation index of a process $X$ is
\begin{equation*}
\mathrm{var}_T(X) = \sup\left\{ p\,;\, V_T^p(X) < \infty\right\}.
\end{equation*}

For the fBm we have
\begin{equation*}
V^p_T(B^H)=\left\lb
\begin{array}{lll}
\infty & ; & pH<1\\
T\mu_p & ; & pH=1\\
0 & ; & pH>1
\end{array}
\right.
\end{equation*}
where $\mu_p$ is the $p$th moment of a standard Gaussian random variable, see \cite{Dudley-Norvaisa-1998, Dudley-Norvaisa-1999}. Consequently, $\mathrm{var}_T(B^H) = 1/H$.

While the fBm has been proposed as a model for financial time series, modeling with it makes arbitrage possible, see \cite{Bender-Sottinen-Valkeila-2007}. To eliminate this problem, a generalization called mixed fractional Brownian motion (mfBm) was introduced in \cite{Cheridito-2001}.  This is the mixture model
$$
M^{a,b} = aB + bB^H,
$$ 
where $a,b\in\R$ and $B$ is a standard Brownian motion (Bm) independent of the fBm $B^H$. If $H>1/2$, the mfBm has the path roughness governed by the Bm part and the long-range dependence governed by the fBm part.  Hence, e.g., in pricing of financial derivatives the corresponding mixed Black--Scholes model yields the same option prices as the standard Brownian model, see \cite{Bender-Sottinen-Valkeila-2008}.

A natural generalization of the mfBm is to consider two (or $n$) independent fBm mixtures, see \cite{Almani-hosseini-Tahmasebi-2021}. In this paper, we study an independent infinite-mixture generalization that we call the multi-mixed fractional Brownian motion (mmfBm) with parameters $\sigma_k$, $H_k$, $k\in\N$:
$$
M = \sum_{k=1}^\infty \sigma_k B^{H_k},
$$
where $B^{H_k}$'s are independent fBm's with Hurst indices $H_k\in(0,1)$, and $\sigma_k$'s are positive volatility constants satisfying $\sum_{k=1}^\infty \sigma_k^2 < \infty$.

\medskip

For other kinds of generalizations of the fBm, see e.g. \cite{Houdre-Villa-2003,Levy-Vehel-1995,Perrin-Harba-Berzin-Joseph-Iribarren-Bonami-2001,Perrin-Harba-Iribarren-Jennane-2005}.

\bigskip

The fractional Ornstein--Uhlenbeck process (fOU) $U^{\lambda,H}$, with parameters $\lambda>0$ and $H\in(0,1)$ is the stationary solution of the Langevin equation
$$
\d U^{\lambda,H}_t = -\lambda U^{\lambda,H}_t \d t + \d B^H_t,
$$
which is given by
$$
U^{\lambda,H}_t = \int_{-\infty}^t \e^{-\lambda(t-s)}\, \d B^H_s,
$$
where $(B^H_s)_{s\le 0}$ is an independent copy of the fBm $(B^H_s)_{s\ge 0}$, see \cite{Cheridito-Kawaguchi-Maejima-2003}.
Note that the Langevin equation and its solution can be understood via integration-by-parts.
As a stationary process, the fOU admits the spectral density 
\begin{equation}\label{fou-sd}
f_{\lambda,H}(x)
=
\frac{f_H(x)}{x^2+ \lambda^2},
\end{equation}
where $f_H$ is the spectral density of the driving fBm \eqref{fbm-sd}, see \cite{Barboza-Viens-2017}.
Denote, for $\alpha\in (-1,0)\cup (0,1)$
\begin{eqnarray}
{\gamma}_\alpha(x) &=& \frac{1}{\Gamma(\alpha)}\int_0^x s^{\alpha-1}\e^{s}\, \d s,\label{NLG} \\ 
{\Gamma}_\alpha(x) &=& \frac{1}{\Gamma(\alpha)}\int_x^\infty s^{\alpha-1}\e^{-s}\, \d s,\label{NUG} 
\end{eqnarray}	
and ${\gamma}_0(x)=1$, ${\Gamma}_0(x)=0$.
The autocovariance function of the fOU process can be written as
\begin{equation}\label{fou-cov}
\rho_{\lambda,H}(t)
=
\frac{\Gamma(1+2H)}{4}
\frac{\e^{-\lambda t}}{\lambda^{2H}}
\bigg\lb 1+{\gamma}_{2H-1}(\lambda t)+\e^{2\lambda t}{\Gamma}_{2H-1}(\lambda t)\bigg\rb,
\end{equation}
see Proposition  \ref{pro-4} below. For $H=\frac12$ we recover the well-known Bm case 
\begin{equation*}
\rho_{\lambda,\frac12}(t)=\frac{\e^{-\lambda t}}{2\lambda}.
\end{equation*}

For $t\to\infty$ we have the power decay
$$
\rho_{\lambda ,H}(t)=\frac12\sum_{n=1}^N\lambda^{-2n}\left(\prod_{j=0}^{2n-1}(2H-j)\right)t^{2H-2n} + O(t^{2H-2N-2}),
$$
for $N=1,2,\ldots$, i.e., the fOU process with $H>\frac12$ is long-range dependent, and for $H\le\frac12$ it is short-range dependent, see \cite{Cheridito-Kawaguchi-Maejima-2003}.

\medskip

The H\"older index and the $p$-variation $p$-variation index of fOU is the same as for the fBm: $\mathrm{Hol}_T(U^{\lambda,H}) = \mathrm{Hol}_T(B^H)=H$ and
$\mathrm{var}_T(U^{\lambda,H}) = \mathrm{var}_T(B^H) =1/H$.  These results follow e.g. from our Theorem \ref{thm 1.3} and Theorem \ref{thm:p-v}.

\medskip
In this paper we study the multi-mixed fractional Ornstein--Uhlenbeck process (mmfOU) with parameters $\lambda>0$ and $\sigma_k$, $H_k$, $k\in\N$, that is defined naturally as the stationary solution of Langevin equation with mmfBm as the driving noise:
$$
\d U_t = -\lambda U_t\,\d t + \d M_t,
$$ 
with
$$
U_0 = \int_{-\infty}^0 \e^{\lambda s}\, \d M_s,
$$
where $(M_s)_{s\le 0}$ is an independent copy of the mmfBm.

The rest of the paper is organized as follows.
In Section \ref{sect:definitions} we define the multi-mixed fractional Brownian motions (mmfBm) and the associated multi-mixed fractional Ornstein--Uhlenbeck (mmfOU) processes, prove their existence in $L^2(\Omega\times [0,T])$, and provide their basic properties.
The long-range dependence of these processes are what we studied in Section \ref{sect:LRD}
In Section \ref{sect:Continuity} we analyze the H\"older continuity and $p$-variation of mmfBm's and mmfOU processes.
The $p$-variation of these processes are calculated in Section \ref{sect:variation}
In Section \ref{sect:CFS} we show that the mmfBm's and mmfOU processes have the conditional full support property.
Finally, In Section \ref{sect:SP} some simulated path of these processes are given.

\section{Definitions and Basic Properties}\label{sect:definitions}

\begin{dfn}
Let $\sigma_k$, $k\in\N$, satisfy
\begin{equation}\label{ass1}
\sum_{k=1}^\infty \sigma_k^2 < \infty,
\end{equation} 	
and let $H_k$, $k\in\N$, satisfy
\begin{eqnarray} \nonumber
H_k\ne H_l \mbox{ for } k\ne l, \\ \label{ass2}
H_{\inf} = \inf_{k\in\N} H_k > 0 \\ \nonumber
H_{\sup} = \sup_{k\in\N} H_k <1.
\end{eqnarray}
The multi-mixed fractional Brownian motion (mmfBm) is 
$$
M = \sum_{k=1}^\infty \sigma_k B^{H_k},
$$
where $B^{H_k}$, $k\in\N$, are independent fBm's.
\end{dfn}	

The following proposition shows the existence of the mmfBm.

\begin{pro}
The mmfBm $M$ exist as a random function taking values in $L^2(\Omega\times[0,T])$ for all $T>0$.	
\end{pro}	

\begin{proof}
Let 
$
M^n = \sum_{k=1}^n \sigma_k B^{H_k}.
$	Clearly $M^n$ takes values in $L^2(\Omega\times[0,T])$.
Let $n,m\in\N$ with $n>m$. Then
\begin{eqnarray*}
{\| M^{n} - M^{m}\|}^2_{L^2(\Omega\times[0,T])}
&=&
\int_0^T \E\left[(M_t^n - M_t^m)^2\right]\, \d t\\
&=&
\int_0^T \E\left[\left(\sum_{k=m+1}^n \sigma_k B^{H_k}_t\right)^2\right] \, \d t \\
&=&
\sum_{k=m+1}^n \int_0^T \sigma_k^2 \E\left[(B_t^{H_k})^2\right] \, \d t \\
&=&
\sum_{k=m+1}^n \int_0^T \sigma_k^2 t^{2H_k} \, \d t  \\
&=&
\sum_{k=m+1}^n \sigma_k^2 \frac{T^{1+2H_k}}{1+2H_k} \\
&\le&
\sum_{k=m+1}^n \sigma_k^2\, \max\left\{1,T^{3}\right\},
\end{eqnarray*}	
which shows that the sequence $(M^n)_{n\in\N}$ is Cauchy. Thus $M^n\to M$ in
$L^2(\Omega\times[0,T])$ showing the existence.
\end{proof}	

In the same way we see that the mmfBm $(M_t)_{t\ge0}$ exist in the sense that $M^n_t \to M_t$ in $L^2(\Omega)$ for all $t\ge0$.

The following is now obvious:

\begin{pro}
The mmfBm has stationary increments, its covariance function is
\begin{equation}\label{mmfbm-cov}
r(t,s) = \sum_{k=1}^\infty \sigma_k^2 r_{H_k}(s,t) = \frac12\sum_{k=1}^\infty \sigma_k^2 \left[ |t|^{2H_k} + |s|^{2H_k} - |t-s|^{2H_k}\right],
\end{equation}
and it admits the spectral density
\begin{equation}\label{mmfbm-sd}
f(x) = \sum_{k=1}^\infty \sigma_k^2 f_{H_k}(x) = \sum_{k=1}^\infty 
\frac{\sin(\pi H_k)\Gamma(1+2H_k)}{2\pi}\sigma_k^2|x|^{1-2H_k}.
\end{equation}	
\end{pro}	

\begin{dfn}
The multi-mixed fractional Ornstein--Uhlenbeck process (mmfOU) $U$ with parameter $\lambda>0$, is the stationary solution of the Langevin equation
\begin{equation}
\d U_t = -\lambda U_t \d t + \d M_t,\label{Langevin}
\end{equation}
where the equation is understood in the integration-by-parts sense.
\end{dfn}	

\begin{pro}\label{pro-3}
On $L^2(\Omega\times [0,T])$, the mmfOU can be represented as the integral
$$
U_t = \e^{-\lambda t}\xi +\int_{0}^t \e^{-\lambda(t-s)}\, \d M_s,
$$	
where the integral is understood in the integration-by-parts sense, and
$$
\xi = \int_{-\infty}^0 \e^{\lambda s}\, \d M_s, 
$$ 
where $(M_s)_{s\le 0}$ is an independent copy of the mmfBm $(M_s)_{s\ge0}$.
\end{pro}		

\begin{proof}
Let $M^n = \sum_{k=1}^n \sigma_k B^{H_k}$. Then, the stationary solution of the Langevin equation
$$
\d U^n_t = -\lambda U^n_t \d t + \d M^n_t,
$$
is given by 
$$
U^n_t = \e^{-\lambda t}\xi_n +\int_{0}^t \e^{-\lambda(t-s)}\, \d M^n_s,
$$
where
$$
\xi_n = \int_{-\infty}^0 \e^{\lambda s}\, \d M^n_s. 
$$
Then, with integration-by-parts
\begin{equation*}
\int_{0}^t \e^{\lambda s}\, \d M^n_s = \e^{\lambda t}M^n_t - \lambda\int_{0}^t \e^{\lambda s}M^n_s\, \d s
\ \to\ \e^{\lambda t}M_t - \lambda\int_{0}^t \e^{\lambda s}M_s\, \d s=\int_{0}^t \e^{\lambda s}\, \d M_s,
\end{equation*}
because $M^n\to M$ in $L^2(\Omega\times[0,T])$. With the same arguments $\xi_n \to \xi$ in $L^2(\Omega)$. This yields $U^n \to U$ in $L^2(\Omega\times[0,T])$.
\end{proof}	

The following technical lemma is used to calculate spectral densities.

\begin{lmm}\label{lmm-1}
For $0\ne p\in(-1,1),\,\lambda >0,\, t>0$
\begin{equation}\label{int_cont}
\int_{-\infty}^\infty\e^{\i tx}\frac{|x|^p}{\lambda^2+x^2}\, \d x =
\frac{\pi\e^{-\lambda t}}{2\cos(\frac{p\pi}{2})\lambda^{1-p}}
\left\lb 1 + \gamma_{-p}(\lambda t) + \e^{2\lambda t}\Gamma_{-p}(\lambda t)\right\rb,
\end{equation}
where $\gamma_{-p}$ and $\Gamma_{-p}$ are given by \eqref{NLG} and \eqref{NUG}.
\end{lmm}

\begin{proof}
Recall that for the Fourier transform
\begin{equation*}
\mathscr{F}(f)(x)=\frac{1}{\sqrt{2\pi}}\int_{-\infty}^\infty \e^{-\i tx}f(t)\, \d t,
\end{equation*}
we have the convolution theorem
\begin{equation}
\int_{-\infty}^\infty \e^{\i tx}\mathscr{F}(f)(x)\mathscr{F}(g)(x)\, \d x=\int_{-\infty}^\infty f(t-\xi)g(\xi)\, \d\xi. \label{convolution}
\end{equation}
Moreover, we have
\begin{align}
&\mathscr{F}\left(e^{-\lambda|t|}\right)=\sqrt{\frac{2}{\pi}}\cdot\frac{\lambda}{\lambda^2+x^2},\label{F_1}\\
&\mathscr{F}\left(|t|^{\alpha}\right)=\sqrt{\frac{2}{\pi}}\cdot\Gamma(\alpha+1)\cos\left(\frac{(\alpha+1)\pi}{2}\right)|x|^{-(\alpha+1)}.\label{F_2}
\end{align}
The first formula \eqref{F_1} is valid for $\lambda>0$. The second formula \eqref{F_2} is valid for $-1<\alpha<0$. For $-2<\alpha<-1$, because of the function $|t|^{\alpha}$, some singular terms arise at the origin.  Nevertheless, it admits a unique meromorphic extension as a tempered distribution, also denoted $|t|^{\alpha}$ as a homogeneous distribution on all real line $\R$ including the origin (see \cite{gelfand1964generalized}). So, we use that extension and formula \eqref{F_2} will be valid for all $-1\ne\alpha\in(-2,0)$. So, using $f(t)=e^{-\lambda|t|}$ and $g(t)=|t|^\alpha$ in \eqref{convolution} we obtain
\begin{eqnarray*}
\lefteqn{\frac{2}{\pi}\cdot\Gamma(\alpha+1)\cos\left(\frac{(\alpha+1)\pi}{2}\right)\lambda
\int_{-\infty}^\infty\e^{itx}\frac{|x|^{-(\alpha+1)}}{\lambda^2+x^2}\, \d x} \\ &=&\int_{-\infty}^\infty |\xi|^\alpha\e^{-\lambda |t-\xi|}\, \d\xi\\
&=&\int_{-\infty}^0 (-\xi)^\alpha\e^{-\lambda (t-\xi)}\, \d\xi
+ \int_0^t\xi^\alpha\e^{-\lambda (t-\xi)}\, \d\xi
+ \int_t^\infty\xi^\alpha\e^{-\lambda (\xi-t)}\, \d\xi\\
&=&\frac{\e^{-\lambda t}}{\lambda^{(\alpha+1)}}
\int_0^{\infty} u^\alpha\e^{-u}\, \d u
+ \frac{\e^{-\lambda t}}{\lambda^{(\alpha+1)}}
\int_0^{\lambda t} u^\alpha\e^{u}\, \d u 
+ \frac{\e^{\lambda t}}{\lambda^{(\alpha+1)}}
\int_{\lambda t}^\infty u^\alpha\e^{-u}\, \d u\\
&=&\frac{\e^{-\lambda t}\Gamma(\alpha+1)}{\lambda^{(\alpha+1)}}
\left\lb 1 + \gamma_{(\alpha+1)}(\lambda t) + \e^{2\lambda t}\Gamma_{(\alpha+1)}(\lambda t)\right\rb.
\end{eqnarray*}
Now, choosing $p=-(\alpha+1)$ proves \eqref{int_cont}.
\end{proof}
It follows from Lemma \ref{lmm-1} that:
\begin{pro}\label{pro-4}
The covariance function of the fOU is
\begin{equation} \label{fou-cov}
\rho_{\lambda,H}(t) = \E[U^{\lambda,H}_{s}U^{\lambda,H}_{s+t}] = 
\frac{\Gamma(1+2H)}{4} \frac{\e^{-\lambda t}}{\lambda^{2H}}
\bigg\lb 1+{\gamma}_{2H-1}(\lambda t)+\e^{2\lambda t}{\Gamma}_{2H-1}(\lambda t)\bigg\rb.
\end{equation}
\end{pro}
\begin{pro}\label{pro-5}
The covariance function of the mmfOU is
\begin{equation} \label{mmfou-cov}
\rho_\lambda(t)= \E[U_{s}U_{s+t}] = 
\sum_{k=1}^\infty\sigma_k^2\frac{\Gamma(1+2H_k)\e^{-\lambda t}}{4\lambda^{2H_k}}
\bigg\lb 1+\gamma_{2H_k-1}(\lambda t)+\e^{2\lambda t}\Gamma_{2H_k-1}(\lambda t)\bigg\rb,
\end{equation}
and it admits the spectral density
\begin{equation}\label{mmfou-sd}
f_\lambda(x) = \sum_{k=1}^\infty\sigma_k^2\frac{\sin(\pi H_k)\Gamma(1+2H_k)}{2\pi}
\frac{|x|^{1-2H_k}}{x^2 + \lambda^2}.
\end{equation}	
\end{pro}	

\begin{proof}
Let $U^n$ be like in the proof of Proposition \ref{pro-3}, then
$$
f_{\lambda, n}(x) = 
\sum_{k=1}^n 
\sigma_k^2\frac{\sin(\pi H_k)\Gamma(1+2H_k)}{2\pi}
\frac{|x|^{1-2H_k}}{x^2 + \lambda^2},
$$
and $f_{\lambda, n}(x)\to f_\lambda(x)$ because $U^n \to U$ in $L^2(\Omega\times[0,T])$. This proves \eqref{mmfou-sd}. Similarly, \eqref{mmfou-cov} follows by Proposition \ref{pro-4}.
\end{proof}

\begin{rem}
Proposition \ref{pro-4} represents the covariance function $\rho_{\lambda,H}(t)$ in a form involving special functions. However, these special complex functions are usually not suitable for numerical computations. For example,in \cite{Barboza-Viens-2017}, Lemma B.1, the following representation was used for $H>\frac12$
\begin{eqnarray*}
\rho_{\lambda,H}(t) &=& H\Gamma(2H)\frac{e^{-\lambda t}}{\lambda^{2H}}\left\lbrace\frac{1+e^{2\lambda t}}{2} -
\frac{\lambda}{\Gamma(2H-1)}I_{\lambda,H}(t)\right\rbrace,\\
I_{\lambda,H}(t) &=& \int_0^t\int_0^{\lambda v}e^{2\lambda v}e^{-s}s^{2H-2}\,\d s\d v.
\end{eqnarray*}
the double integral above seems reasonable enough, but yields slow numerical calculation in practice. This can be remedied by calculating the inner integral as follows:
\begin{eqnarray*}
I_{\lambda,H}(t) &=& \int_0^{\lambda t}\int_{s/{\lambda}}^{t}e^{2\lambda v}e^{-s}s^{2H-2}\, \d s \d v\\
&=&\frac{1}{2\lambda}\int_0^{\lambda t}s^{2H-2}(e^{2\lambda t-s}-e^s)\,\d s\\
&=&\frac{e^{\lambda t}}{\lambda}\int_0^{\lambda t}s^{2H-2}\sinh(\lambda t-s)\, \d s.
\end{eqnarray*}
Consequently, 
\begin{equation*}\label{rhoH-empiric>}
\rho_{\lambda,H}(t) = \frac{\Gamma(2H+1)}{2\lambda^{2H}}\bigg\{\cosh(\lambda t)
 -
\frac{1}{\Gamma(2H-1)}\int_0^{\lambda t}s^{2H-2}\sinh(\lambda t-s)\, ds\bigg\},
\end{equation*}
\end{rem}
For the case $H<1/2$ we use the following developed version of Lemma 5.1 in \cite{hu2010parameter} for $\alpha>-1$. The proof is similar.
\begin{lmm}\label{2integ}
For $\alpha>-1$
\begin{equation*}
\int_0^\infty\int_0^\infty\e^{-(x+y)}|x-y|^\alpha\, \d x\d y = \Gamma(\alpha + 1).
\end{equation*}
\end{lmm}
\begin{thm}\label{thm:tho-empiric}
For the fOU process $U^{\lambda ,H}$. We have
\begin{eqnarray}\label{rhoH-empiric<}
\rho_{\lambda,H}(t) = \frac{\Gamma(2H+1)}{2\lambda^{2H}}\,\bigg\lbrace\cosh(\lambda t)
-
\frac{1}{\Gamma(2H)}\int_0^{\lambda t}s^{2H-1}\cosh(\lambda t-s)\, ds\bigg\rbrace,\notag
\end{eqnarray}
and so for mmfOU process we have
\begin{eqnarray*}\label{rho-empiric}
\rho_{\lambda}(t) = \sum_{k=0}^\infty\sigma_k^2\frac{\Gamma(2H_k+1)}{2\lambda^{2H_k}}\,\bigg\{\cosh(\lambda t) 
-
\frac{1}{\Gamma(2H_k)}\int_0^{\lambda t}s^{2H_k-1}\cosh(\lambda t-s)\, ds\bigg\}.
\end{eqnarray*}
\end{thm}

\begin{proof}
For $H=1/2$ the right hand side of \eqref{rhoH-empiric<} is $\e^{-\lambda t}/2\lambda$ wich is equal to the outocovariance of the classical Ornstein--Uhlenbeck process with respect to the standard Brownian motion. For $H>1/2$, we obtain \eqref{rhoH-empiric<} from \eqref{rhoH-empiric>} via integration by part. To prove it for $H<1/2$, we will apply the same approach of the proof of {\it Lemma B.1} in \cite{Barboza-Viens-2017}
\begin{align*}
\rho_{\lambda,H}(t) =&\ \E [U^{\lambda ,H}_tU^{\lambda ,H}_0]\\
=&\ \E\left[ \int_{-\infty}^0 \e^{\lambda u}\, \d B^H_u \int_{-\infty}^t \e^{-\lambda(t-v)}\, \d B^H_v\right] \\
=&\ \e^{-\lambda t}\bigg\lbrace\Var(U^{\lambda ,H}_0)
+ 
\E\bigg[ \int_{-\infty}^0 \e^{\lambda u}\, \d B^H_u \int_0^t \e^{\lambda v}\, \d B^H_v\bigg]\bigg\rbrace.
\end{align*}
To obtain the term $\Var(U^{\lambda ,H}_0)$ in a close form, \cite{Barboza-Viens-2017} referred to {\it Lemma 5.2} in \cite{hu2010parameter}; 
however, it was only obtained for $H\geq1/2$, and so we need to extend their result for $H<1/2$.

Since
\begin{equation*}
U^{\lambda ,H}_0 = \int_{-\infty}^0 \e^{\lambda u}\, \d B^H_u = -\lambda\int_{-\infty}^0 \e^{\lambda u}B^H_u\,\d u,
\end{equation*}
we have
\begin{align*}
\Var(U^{\lambda ,H}_0) =&\Var\left[-\lambda\int_{-\infty}^0 \e^{\lambda u}B^H_u\,\d u\right]\\
=&\ \lambda^2\Var\left[\int_0^{\infty} \e^{-\lambda u}B^H_u\,\d u\right]\\
=&\ \lambda^2\E\left[\left(\int_0^{\infty} \e^{-\lambda u}B^H_u\,\d u\right)^2\right]\\
=&\ \lambda^2\E\left[\int_0^{\infty}\int_0^{\infty}\e^{-\lambda (u+v)}B^H_u B^H_v\,\d u\d v\right]\\
=&\ \frac{\lambda^2}{2}\int_0^{\infty}\int_0^{\infty}\e^{-\lambda (u+v)}\cdot\Big\lbrace u^{2H}+v^{2H}-|u-v|^{2H}\Big\rbrace\,\d u\d v\\
=&\ \frac{\lambda^2}{2}\bigg\lbrace 2\left(\int_0^{\infty}\e^{-\lambda u}\,\d u\right)\left(\int_0^{\infty}\e^{-\lambda v} v^{2H}\,\d v\right)
-\int_0^{\infty}\int_0^{\infty}\e^{-\lambda (u+v)}|u-v|^{2H}\,\d u\d v\bigg\rbrace,
\end{align*}
Now choosing $x=\lambda u,\,y=\lambda v$ and Lemma \ref{2integ} we have
\begin{align}
\Var (U^{\lambda ,H}_0) &=\frac{\lambda^{-2H}}{2}\bigg\lbrace 2\int_0^{\infty}\e^{-y} y^{2H}\,\d y
-\int_0^{\infty}\int_0^{\infty}\e^{-(x+y)}|x-y|^{2H}\,\d x\d y\bigg\rbrace\notag\\
&=\frac{\lambda^{-2H}}{2}\Big[ 2\Gamma(2H+1)-\Gamma(2H+1)\Big]\notag\\
&=\lambda^{-2H}H\Gamma(2H).\label{Var(U0)}
\end{align}
On the other hand, as in Lemma 2.1 in \cite{Cheridito-Kawaguchi-Maejima-2003} and the proof of Lemma B.1 in \cite{Barboza-Viens-2017}, using formula 
\begin{equation*}
\gamma_\ell(z,x) = \frac{\gamma_\ell(z+1,x)+x^z\e^{-x}}{z},
\end{equation*}
where $\gamma_\ell$ is the well-known lower Gamma function, for $H<1/2$ we have
\begin{align}
&\E\left[ \int_{-\infty}^0 \e^{\lambda u}\, \d B^H_u \int_0^t \e^{\lambda v}\, \d B^H_v\right]\label{E_int}\\
&=H(2H-1)\int_{-\infty}^0\int_0^t\e^{-\lambda (u+v)}|u-v|^{2H-2}\,\d u\d v\notag\\
&=\Var (U^{\lambda ,H}_0)\bigg\lbrace \frac{\e^{2\lambda t}-1}{2} -\frac{\lambda}{\Gamma(2H-1)}\int_0^t\e^{2\lambda v}\int_0^{\lambda v}\e^{-s}s^{2H-2}\,\d s\d v\bigg\rbrace\notag\\
&=\Var (U^{\lambda ,H}_0)\bigg\lbrace \frac{\e^{2\lambda t}-1}{2} -\frac{\lambda}{\Gamma(2H-1)}\int_0^t\e^{2\lambda v}\gamma_\ell(2H-1,\lambda v)\,\d v\bigg\rbrace\notag\\
&=\Var (U^{\lambda ,H}_0)\bigg\lbrace \frac{\e^{2\lambda t}-1}{2} -\frac{\lambda}{\Gamma(2H)}\int_0^t\e^{2\lambda v}\gamma_\ell(2H,\lambda v)\,\d v - \frac{\lambda^{2H}}{\Gamma(2H)}\int_0^t\e^{\lambda v}v^{2H-1}\,\d v\bigg\rbrace\notag\\
&=\Var (U^{\lambda ,H}_0)\bigg\lbrace \frac{\e^{2\lambda t}-1}{2} -\frac{\lambda}{\Gamma(2H)}\int_0^t\e^{2\lambda v}\int_0^{\lambda v}\e^{-s}s^{2H-1}\,\d s\d v - \frac{\lambda^{2H}}{\Gamma(2H)}\int_0^t\e^{\lambda v}v^{2H-1}\,\d v\bigg\rbrace.\notag
\end{align}
Using \eqref{Var(U0)} and \eqref{E_int}, with similar arguments as we did for \eqref{rhoH-empiric>} we obtain \eqref{rhoH-empiric<}. 
\end{proof}

\section{Long-Range Dependence}\label{sect:LRD}
The increments of fBm, the fGn,  is a well-known stationary process, that is long-range dependent (LRD) if $H>1/2$, and short-range dependent (SRD) in $H<\frac12$
Motivated with this, we consider the LRD of the increments the of mmfBm. 

For a lag $\delta$ and a process $X$ we denote
$
\Delta_\delta X_t = X_{t+\delta} - X_t.
$
Then
\begin{equation*}
\Delta_\delta M_t = \sum_{k=1}^\infty\sigma_k \Delta_\delta B^{H_k}_t
\end{equation*}
is stationary and its autocovariance function is denoted by 
\begin{equation*}
\varrho(\delta ;t)=\E\big[\Delta_\delta M_{s+t}\Delta_\delta M_s\big],
\end{equation*}

\begin{thm}\label{LRD_mmfBm}
For $t\to\infty$
\begin{equation}\label{asymp_mmfBm}
\varrho(\delta ;t)\sim \delta^2\sum_{k=1}^\infty\sigma_k^2H_k(2H_k-1)t^{2H_k-2}=O(t^{2H_{\sup}-2}).
\end{equation}
So the mmfBm increment process $\Delta_\delta M_t$ is LRD if and only if $H_k>1/2$ for some $k\geq 0$.
\end{thm}

\begin{proof}
By using the generalized binomial theorem
\begin{eqnarray}
\varrho(\delta ;t) &=& \frac12\sum_{k=1}^\infty\sigma_k^2\Big\lbrace (t+\delta)^{2H_k} + (t-\delta)^{2H_k} - 2t^{2H_k}\Big\rbrace\nonumber\\
&=& \frac12\sum_{k=1}^\infty\sigma_k^2t^{2H_k}\Bigg\lbrace \Big(1+\frac{\delta}{t}\Big)^{2H_k} + \Big(1-\frac{\delta}{t}\Big)^{2H_k} - 2\Bigg\rbrace\nonumber\\
&=& \frac12\sum_{k=1}^\infty\sigma_k^2t^{2H_k}\Bigg\lbrace \sum_{r=0}^\infty\binom{2H_k}{r}\Big(\frac{\delta}{t}\Big)^r + \sum_{r=0}^\infty\binom{2H_k}{r}(-1)^r\Big(\frac{\delta}{t}\Big)^r - 2\Bigg\rbrace\nonumber\\
\label{lrdp}&\sim & \delta^2\sum_{k=1}^\infty\sigma_k^2H_k(2H_k-1)t^{2H_k-2}.
\end{eqnarray}
Since
\[
\sigma_k^2H_k(2H_k-1)t^{2H_k-2}\leq\sigma_k^2,
\]
the series \eqref{lrdp} is uniformly convergent. So we have
\begin{equation*}
\lim_{t\to\infty}\sum_{k=1}^\infty\sigma_k^2H_k(2H_k-1)t^{2H_k-2}=\sum_{k=1}^\infty\lim_{t\to\infty}\sigma_k^2H_k(2H_k-1)t^{2H_k-2}.
\end{equation*}
This yields \eqref{asymp_mmfBm}.
\end{proof}

To investigate LRD for the mmfOU process, we first need some lemmas.

The following theorem shows that similar to the mmfBm increment process, the long-range dependence of the mmfOU is governed by the long-range dependence of the largest Hurst index in the driving mmfBm.
\begin{thm}\label{LRD_thm}
For $t\to\infty$ and each $N=1,2,\ldots$
\begin{equation}
\rho_\lambda(t)=\frac12\sum_{k=1}^\infty\sum_{n=1}^N\sigma^2_k\lambda^{-2n}\left(\prod_{j=0}^{2n-1}(2H_k-j)\right)t^{2H_k-2n} + O(t^{2H_{\sup}-2N-2}).\label{asymp_mmfOU}
\end{equation} 
So the mmfOU process $U$ is LRD if and only if $H_k>1/2$ for some $k\geq 0$.
\end{thm}

\begin{proof}
By the proof of Lemma 2.2 and Theorem 2.3 in \cite{Cheridito-Kawaguchi-Maejima-2003}
\begin{eqnarray}\label{asymp_rho}
\qquad&\rho_\lambda(t)=&\E\left[ \int_{-\infty}^0 \e^{\lambda u}\, \d M_u \int_{-\infty}^t \e^{-\lambda(t-v)}\, \d M_v\right]\\
&=& \e^{-\lambda t}\E\left[ \int_{-\infty}^0 \e^{\lambda u}\, \d M_u \int_{-\infty}^{1/\lambda} \e^{\lambda v}\, \d M_v\right]\notag\\
&+& \e^{-\lambda t}\sum_{i=1}^{\infty}\sigma_i^2 H_i(2H_i-1)
\int_{-\infty}^0 \e^{\lambda u} \left(\int_{1/\lambda}^t \e^{\lambda v}(v-u)^{2H_i-2}\,\d v\right)\d u\notag\\
&=& O(\e^{-\lambda t}) + \frac12\sum_{i=1}^{\infty}\sigma_i^2 \frac{H_i(2H_i-1)}{\lambda^{2H_i}}
\Bigg\lbrace\e^{-\lambda t}\int_1^{\lambda t}\e^y y^{2H_i-2}\,\d y
+  \e^{\lambda t}\int_{\lambda t}^{\infty} \e^{-y} y^{2H_i-2}\,\d y\Bigg\rbrace\notag\\
&\leq & O(\e^{-\lambda t})
+ \frac12\sum_{k=1}^\infty\sum_{n=1}^N\sigma^2_k\lambda^{-2n}\left(\prod_{j=0}^{2n-1}(2H_k-j)\right)t^{2H_k-2n}\notag\\
&+& \frac12\sum_{k=1}^{\infty}\sigma_k^2 \frac{\Big|H_k(2H_k-1)\cdots (2H_k-2-2N)\Big|}{\lambda^{2H_k}}
 \left[\e^{-\frac{\lambda t}{2}} + (1+2^{2H_k-2N-3})(\lambda t)^{2H_k-2N-3}\right].\notag
\end{eqnarray}
Now, for $t\in [1,\infty)$
\begin{eqnarray*}
&\sigma_k^2& \frac{\Big|H_k(2H_k-1)\cdots (2H_k-2-2N)\Big|}{\lambda^{2H_k}}
\,\e^{-\frac{\lambda t}{2}}<\Lambda_N\sigma_k^2\\
&\sigma_k^2& \frac{\Big|H_k(2H_k-1)\cdots (2H_k-2-2N)\Big|}{\lambda^{2H_k}}
\,(1+2^{2H_k-2N-3})(\lambda t)^{2H_k-2N-3}<\Pi_N\sigma_k^2,
\end{eqnarray*}
where
\begin{eqnarray*}
\Lambda_N &=& H_{\sup}\frac{\max\Big(|2H_{\inf}-1|,|2H_{\sup}-1|\Big)}{\max\Big(\lambda^{2H_{\inf}},\lambda^{2H_{\sup}}\Big)}
\Big|(2H_{\inf}-2)\cdots (2H_{\inf}-2-2N)\Big|,\\
\Pi_N &=& H_{\sup}\frac{\max\Big(|2H_{\inf}-1|,|2H_{\sup}-1|\Big)}{\lambda^{2N+3}}
\Big|(2H_{\inf}-2)\cdots (2H_{\inf}-2-2N)\Big|
(1+2^{2H_{\sup}-2N-3}).
\end{eqnarray*}
So, as $\sum_{k=1}^\infty\sigma_k^2<\infty$, the series in the right-hand side of the inequality \eqref{asymp_rho} is uniformly convergent on $t\in [1,\infty)$. Hence
\begin{eqnarray*}
\lim_{t\to\infty}\sum_{k=1}^{\infty}\sigma_k^2 \frac{\Big|H_k(2H_k-1)\cdots (2H_k-2-2N)\Big|}{\lambda^{2H_k}}\left[\e^{-\frac{\lambda t}{2}} + (1+2^{2H_k-2N-3})(\lambda t)^{2H_k-2N-3}\right]\\
=\sum_{k=1}^{\infty}\sigma_k^2 \frac{\Big|H_k(2H_k-1)\cdots (2H_k-2-2N)\Big|}{\lambda^{2H_k}}\lim_{t\to\infty}\left[\e^{-\frac{\lambda t}{2}} + (1+2^{2H_k-2N-3})(\lambda t)^{2H_k-2N-3}\right].
\end{eqnarray*}
This proves \eqref{asymp_mmfOU}.
\end{proof}


\section{Continuity}\label{sect:Continuity}

\begin{thm}\label{thm 1.3}
Both mmfBm and mmfOU have Hölder index $H_{\inf}$.
\end{thm}

\begin{proof}
For $\epsilon >0$ and $|t-s|<1$, the mmfBm satisfies
\begin{equation*}
\E\Big[(M_t-M_s)^2\Big] = \sum_{k=1}^\infty\sigma_k^2 |t-s|^{2H_k}
\leq\bigg(\sum_{k=1}^\infty\sigma_k^2\bigg)|t-s|^{2H_{\inf}-\epsilon}
= C_0|t-s|^{2H_{\inf}-\epsilon},
\end{equation*}
where $C_0:=\sum_{k=1}^\infty\sigma_k^2>0$. Thus, H\"older continuity with exponent $H_{\inf}-\varepsilon$ follows from Theorem 1 of  \cite{Azmoodeh-Sottinen-Viitasaari-Yazigi-2014}. On the other hand, for some $j\geq 1$ we have $H_{\inf}\leq H_j<H_{\inf}+\epsilon$ and so the fBm $B^{H_j}$ is not $(H_{\inf}+\epsilon)$-H\"older continuous. Hence the process $M = \sigma_j B^{H_j}+\sum_{k\neq j} \sigma_k B^{H_k}$ is not $(H_{\inf}+\epsilon)$-H\"older continuous. This proves the claim for mmfBm.

For the mmfOU, we apply the Corollary 2 of \cite{Azmoodeh-Sottinen-Viitasaari-Yazigi-2014}. That states that the stationary process $U$ is Hölder-continuous with any exponent $0<a<H_{\inf}$ if and only if for each $0<\epsilon<2H_{\inf}$, there is some $0<\delta <1$ that
\begin{equation}\label{eq 37}
\int_0^\infty (1-\cos(sx))f_{\lambda}(x)dx<C_\epsilon s^{2H_{\inf}-\epsilon},\quad s\in (0,\delta).
\end{equation}
This is equivalent to have
\begin{equation*}
\int_0^\infty \frac{(1-\cos(sx))}{s^{2H_{\inf}-\epsilon}}f_{\lambda}(x)dx<C_\epsilon,\quad s\in (0,\delta).
\end{equation*}
To show this, here for $s<1$ we have
\begin{align*}
&\int_0^\infty \frac{(1-\cos(sx))}{s^{2H_k-\epsilon}}f_{\lambda,H_k}(x)dx\notag\\
&=s^\epsilon c_{H_k}\int_0^\infty (1-\cos(sx))\frac{x\cdot (sx)^{-2H_k}}{\lambda^2+x^2}dx\notag\\
&=s^\epsilon c_{H_k}\int_0^\infty (1-\cos u)\frac{u^{1-2H_k}}{s^2\lambda^2+u^2}du\quad (u=sx)\notag\\
&\leq c_{H_k}\int_0^\infty (1-\cos u)\frac{u^{1-2H_k}}{s^2\lambda^2+u^2}du\quad (0<s<1)\notag\\
&\leq c_{H_k}\Big\lbrace\int_0^\epsilon (1-\cos u)\frac{u^{1-2H_k}}{u^2}du + \int_\epsilon^\infty \frac{u^{1-2H_k}}{u^2}du\Big\rbrace\notag\\
&=c_{H_k}\Big\lbrace\int_0^\epsilon\frac{2\sin^2(\frac{u}{2})}{u^2}u^{1-2H_k}du +\int_\epsilon^\infty u^{-1-2H_k}du\Big\rbrace\notag\\
&\leq c_{H_k}\Big\lbrace\int_0^\epsilon\frac{1}{2}u^{1-2H_k}du +\int_\epsilon^\infty u^{-1-2H_k}du\Big\rbrace\notag\\
&=c_{H_k}\Big\lbrace\frac{\epsilon^{2-2H_k}}{4(1-H_k)}+\frac{\epsilon^{-2H_k}}{2H_k}\Big\rbrace =:C_{\epsilon,H_k}<\infty.
\end{align*}
Therefore,
\begin{equation}
\int_0^\infty (1-\cos(sx))f_{\lambda,H_k}(x)dx\leq C_{\epsilon,H_k}s^{2H_k-\epsilon}\leq C_{\epsilon,H_k}s^{2H_{\inf}-\epsilon}.\label{eq 42}
\end{equation}
Also, we have
\begin{align}
\sum_{k=1}^\infty \sigma_k^2 C_{\epsilon,H_k}&=\sum_{k=1}^\infty\sigma_k^2\, \frac{\sin(\pi H_k)\Gamma(1+2H_k)}{2\pi}\Big\lbrace\frac{\epsilon^{2-2H_k}}{4(1-H_k)}+\frac{\epsilon^{-2H_k}}{2H_k}\Big\rbrace\notag\\
&\leq\frac{\Gamma(3)}{2\pi}\Big\lbrace\frac{\epsilon^{2-2H_{sup}}}{4(1-H_{sup})}+\frac{\epsilon^{-2H_{inf}}}{2H_{inf}}\Big\rbrace\Big(\sum_{k=1}^\infty\sigma_k^2\Big)=:C_\epsilon<\infty,\label{eq 43}
\end{align}
if and only if $0<H_{\inf}\leq H_{\sup}<1$ and $\sum_{k=1}^{\infty}\sigma_k^2<\infty$. Now, \eqref{eq 42} and \eqref{eq 43} yield \eqref{eq 37}. Moreover, for some $j\geq 1$ we have $H_{\inf}\leq H_j<H_{\inf}+\epsilon$ and so the fOU $U^{H_j}$ is not $(H_{\inf}+\epsilon)$-H\"older continuous. Hence the process $U = \sigma_j U^{H_j}+\sum_{k\neq j} \sigma_k U^{H_k}$ is not $(H_{\inf}+\epsilon)$-H\"older continuous. This proves the claim for mmfOU.
\end{proof}

\section{$p$-Variation}\label{sect:variation}

\begin{thm}\label{thm:p-v}
For $p>0$, the equidistant $p$-variations of the mmfBm $M$ and the mmfOU $U$ on the time-interval $[0,T]$ are equal and
\begin{equation}\label{eq:p-v}
V^p_T(M)=V^p_T(U)=
\left\lb
\begin{array}{lll}
\infty & ; & pH_{\inf}<1\\\\
T\Bigg(\displaystyle\underset{H_i=H_{\inf}}{\sum}\sigma_i^2\Bigg)^{p/2}\mu_p & ; & pH_{\inf}=1\\\\
0 & ; & pH_{\inf}>1
\end{array}
\right.
\end{equation}
\end{thm}

\begin{proof} For the mmfBm $M$, We have
\begin{eqnarray*}
v^p_{\pi_n}(M)&:=&\sum_{t_k\in\pi_n}|\Delta M_{t_k}|^p\\
&=&\sum_{t_k\in\pi_n}\left|\sum_{i=1}^{\infty}\sigma_i^2(\Delta t_k)^{2H_i}\right|^{p/2}\left|\frac{\Delta M_{t_k}}{\Big[\sum_{i=1}^{\infty}\sigma_i^2(\Delta t_k)^{2H_i}\Big]^{1/2}}\right|^p\\
&\overset{d}{=}&\left(\sum_{i=1}^{\infty}\sigma_i^2 T^{2H_i}n^{2/p-2H_i}\right)^{p/2}\cdot\frac{1}{n}\sum_{k=1}^n|Z_k|^p.
\end{eqnarray*}
as $|\pi_n|\to 0$ or equivalently $n\to\infty$. Here $(Z_k)$ is a stationary Gaussian process and so by the proof of Lemma 3.7 in \cite{Sottinen-2003}
\begin{equation*}
\frac{1}{n}\sum_{k=1}^n|Z_k|^p\to\mu_p,
\end{equation*}
as $n\to\infty$, where $\mu_p$ is the $p$th absolute moment of the standard Gaussian process. Now, if $pH_{\inf}<1$ then $H_{\inf}<1/p$, and so there exists some $j\geq 1$ that $H_j<1/p$, and so $2/p-2H_j>0$. Therefore
\begin{equation*}
v^p_{\pi_n}(M)\geq\left(\sigma_j^2 T^{2H_j}n^{2/p-2H_j}\right)^{p/2}\cdot\frac{1}{n}\sum_{k=1}^n|Z_k|^p\to\infty.
\end{equation*}
On the other hand, if $pH_{\inf}\geq 1$ for $x\in (1,\infty)$
\begin{equation*}
\sigma_i^2 T^{2H_i}x^{2/p-2H_i}\leq\sigma_i^2T^2,
\end{equation*}
and because $\sum_{i=1}^{\infty}\sigma_i^2<\infty$, the $\sum_{i=1}^{\infty}\sigma_i^2 T^{2H_i}x^{2/p-2H_i}$ is uniformly convergent on $x\in [1,\infty)$. So for $pH_{\inf}\geq 1$
\begin{equation*}
\lim_{n\to\infty}\sum_{i=1}^{\infty}\sigma_i^2 T^{2H_i}n^{2/p-2H_i}=\sum_{i=1}^{\infty}\lim_{n\to\infty}\sigma_i^2 T^{2H_i}n^{2/p-2H_i}.
\end{equation*}
This yields the values mentioned in \eqref{eq:p-v} are correct for the $p$-variation of $M$. For the mmfOU $U$, as it is stationary we have
\begin{equation*}
v^p_{\pi_n}(U) :=\sum_{t_k\in\pi_n}|\Delta U_{t_k}|^p
\overset{d}{=}\sum_{k=1}^n\Big(\Var [\Delta U_{t_1}]\Big)^{p/2}|Z_k|^p
=n\Big(\Var [U_{\frac{T}{n}}-U_0]\Big)^{p/2}\cdot\frac{1}{n}\sum_{k=1}^n|Z_k|^p.
\end{equation*}
As $\frac{1}{n}\sum_{k=1}^n|Z_k|^p\to\mu_p$ for $n\to\infty$, problem vanish to find
\begin{equation*}
\lim_{n\to\infty}n\Big(\Var [U_{\frac{T}{n}}-U_0]\Big)^{p/2}.
\end{equation*}
To find it, again because $U$ is stationary, and using the proof of Theorem \ref{thm:tho-empiric} we have
\begin{eqnarray*}
\Var [U_{\frac{T}{n}}-U_0] &=& \Var U_{\frac{T}{n}} + \Var U_0 - 2\,\Cov\Big(U_{\frac{T}{n}},U_0\Big)\\
&=& 2\,\Var U_0 - 2\,\Cov\Big(U_{\frac{T}{n}},U_0\Big)\\
&=& 2\sum_{i=1}^\infty\sigma_i^2\lambda^{-2H_i}H_i\Gamma(2H_i)
-2\sum_{i=1}^\infty\sigma_i^2\frac{\Gamma(2H_i+1)}{2\lambda^{2H_i}}\bigg\lbrace\cosh\Big(\frac{\lambda T}{n}\Big)\\
& & - \frac{1}{\Gamma(2H_i)}\int_0^{\frac{\lambda T}{n}}s^{2H_i-1}\cosh\Big(\frac{\lambda T}{n}-s\Big)\,\d s\bigg\rbrace\\
&=& \sum_{i=1}^\infty\sigma_i^2\frac{\Gamma(2H_i+1)}{\lambda^{2H_i}}\bigg\lbrace 1 - \cosh\Big(\frac{\lambda T}{n}\Big)
+ \frac{1}{\Gamma(2H_i)}\int_0^{\frac{\lambda T}{n}}s^{2H_i-1}\cosh\Big(\frac{\lambda T}{n}-s\Big)\,\d s\bigg\rbrace.
\end{eqnarray*}
For the large values of $n$ the final series in the right hand side above, is uniformly convergent. So, the $\displaystyle\lim_{n\to\infty}$ and $\sum_{i=1}^\infty$ could change places. This yields
\begin{eqnarray*}
\lefteqn{\lim_{n\to\infty}n\Big(\Var [U_{\frac{T}{n}}-U_0]\Big)^{p/2}} \\
&=&\lim_{n\to\infty}\Big(n^{2/p}\,\Var [U_{\frac{T}{n}}-U_0]\Big)^{p/2}\\
&=& \Bigg(\sum_{i=1}^\infty\sigma_i^2\frac{\Gamma(2H_i+1)}{\lambda^{2H_i}}\cdot\lim_{n\to\infty}n^{2/p}\bigg\lbrace 1 - \cosh\Big(\frac{\lambda T}{n}\Big)
+\,\frac{1}{\Gamma(2H_i)}\int_0^{\frac{\lambda T}{n}}s^{2H_i-1}\cosh\Big(\frac{\lambda T}{n}-s\Big)\,\d s\bigg\rbrace\Bigg)^{p/2}.\\
\end{eqnarray*}
Now for $t\to 0$, by the Taylor expansion 
\begin{equation*}
1-\cosh t = -\sum_{r=1}^\infty \frac{t^{2r}}{(2r)!},
\end{equation*} 
and via integration by parts
\begin{equation*}
\int_0^t s^{2H_i-1}\cosh(t-s)\,\d s = \frac{t^{2H_i}}{2H_i} + \frac{1}{2H_i}\int_0^t s^{2H_i}\sinh(t-s)\,\d s.
\end{equation*}
Again for $t\to 0$, by the Taylor expansion
\begin{equation*}
\int_0^t s^{2H_i}\sinh(t-s)\,\d s \leq  \int_0^t t^{2H_i}\sinh t\,\d s
= t^{2H_i+1}\sinh t
= \sum_{r=1}^\infty \frac{t^{2r+2H_i}}{(2r-1)!}.
\end{equation*}
These yield for $t\to 0$
\begin{equation*}
1-\cosh t + \frac{1}{\Gamma(2H_i)}\int_0^t s^{2H_i-1}\cosh(t-s)\,\d s\sim\frac{t^{2H_i}}{2H_i+1}.
\end{equation*}
Therefore
\begin{equation*}
\lim_{n\to\infty}n\Big(\Var [U_{\frac{T}{n}}-U_0]\Big)^{p/2}=\bigg(\sum_{i=1}^\infty\sigma_i^2T^{2H_i}\lim_{n\to\infty}n^{2/p-2H_i}\bigg)^{p/2},
\end{equation*}
this proves \eqref{eq:p-v}.
\end{proof}

\section{Conditional Full Support}\label{sect:CFS}

As explained in \cite{Bender-Sottinen-Valkeila-2008}, in mathematical finance models one of the must require the so-called Conditional Full Support (CFS) to avoid simple kind of arbitrage. This means that, given the information up to any stopping time $\tau\in [0,T]$, the process is inherently free enough to go anywhere after the stopping time $\tau$ with positive probability. This motivates us to study the CFS property of the mmfBm and mmfOU processes but first we restate the precise definition of CFS from \cite{gasbarra2011conditional}.

\begin{dfn}
Let $X=(X_t)_{0\leq t\leq T}$ be a continuous stochastic process defined on a probability space $(\Omega ,\mathcal{F},\P)$, and $(\mathcal{F}_t)$ be its natural filtration. The process
 $X$ is said to have CFS if, for all $(\mathcal{F}_t)$-stopping times $\tau$, the conditional law of $(X_u)_{\tau\leq u\leq T}$ given $\mathcal{F}_\tau$, almost surely has support $C_{X_\tau}[\tau,T]$, where $C_x[\tau,T]$ is the space of continuous functions $f$ on $[\tau,T]$ satisfying $f(t) = x$. Equivalently, this means that, for all $t\in [0,T]$, $f\in C_0[\tau,T]$, and $\varepsilon>0$,
$$
\P\left(\sup_{\tau\leq u\leq T}|X_u - X_\tau - f(u)|<\varepsilon\bigg\vert\mathcal{F}_\tau\right)>0,
$$
almost surely.
\end{dfn}

\begin{thm}\label{thm 1.4}
Both the mmfBm and the mmfOU have conditional full support.
\end{thm}

\begin{proof}
It is easy to check that
$$
f(x)=\sum_{k=1}^\infty\sigma_k^2f_{H_k}(x)\geq
\left\lbrace
\begin{array}{lll}
\frac{\varepsilon_H\Gamma(1)}{2\pi}\biggl(\displaystyle\sum_{k=1}^\infty\sigma_k^2\biggl)|x|^{1-2H_{\inf}}&:&|x|\leq 1\\
\ \\
\frac{\varepsilon_H\Gamma(1)}{2\pi}\biggl(\displaystyle\sum_{k=1}^\infty\sigma_k^2\biggl)|x|^{1-2H_{\sup}}&:&|x|\geq 1
\end{array}\right.
=:h(x),
$$
where
\begin{equation*}
\varepsilon_H:=\inf \Big\lbrace \sin(\pi H_k)\Big\rbrace_{k\geq 1}=\inf \Big\lbrace \sin(\pi H_{\inf}), \sin(\pi H_{\sup})\Big\rbrace
\end{equation*}
Since $0<H_{\inf}\leq H_{\sup}<1$, $\varepsilon_H>0$. Thus $h(x)>0$ for $x\neq 0$. Thus, for any $x_0>1$ we have
\begin{eqnarray*}\label{eq 19}
\int_{x_0}^\infty\frac{\log f(x)}{x^2} \,\d x
&\ge&
\int_{x_0}^\infty\frac{\log h(x)}{x^2}\, \d x \\
&=&
\log\left\lbrace\frac{\varepsilon_H}{2\pi}\biggl(\sum_{k=1}^\infty\sigma_k^2\biggl)\right\rbrace\int_{x_0}^\infty\frac{\d x}{x^2} + (1-2H_{\sup})\int_{x_0}^\infty\frac{\log x}{x^2}\, \d x  \\
&>& -\infty,
\end{eqnarray*}
and by Theorem 2.1 of \cite{gasbarra2011conditional} this proves that $M$ has conditional full support.

For mmfOU it is easy to check that
\begin{equation*}\label{eq 45}
f_\lambda(x)=\sum_{k=1}^\infty\sigma_k^2 f_{\lambda,H_k}(x)\geq
\left\lbrace
\begin{array}{lll}
\frac{\varepsilon_H\Gamma(1)}{2\pi}\biggl(\displaystyle\sum_{k=1}^\infty\sigma_k^2\biggl)\frac{|x|^{1-2H_{\inf}}}{\lambda^2+x^2}&:&|x|\leq 1\\
\ \\
\frac{\varepsilon_H\Gamma(1)}{2\pi}\biggl(\displaystyle\sum_{k=1}^\infty\sigma_k^2\biggl)\frac{|x|^{1-2H_{\sup}}}{\lambda^2+x^2}&:&|x|\geq 1
\end{array}\right.
=:h(x),
\end{equation*}
where
\begin{equation*}
\varepsilon_H:=\inf \Big\lbrace \sin(\pi H_k)\Big\rbrace_{k\geq 1}=\inf \Big\lbrace \sin(\pi H_{\inf}), \sin(\pi H_{\sup})\Big\rbrace
\end{equation*}
Since $0<H_{\inf}\leq H_{\sup}<1$, we have $\varepsilon_H>0$. Consequently, $h(x)>0$ for $x\neq 0$. Therefore, for any $x_0>1$ we have that
\begin{eqnarray*}
\lefteqn{\int_{x_0}^\infty\frac{\log f_\lambda(x)}{x^2}\, \d x} \\
&\ge&
\int_{x_0}^\infty\frac{\log h(x)}{x^2}\, \d x\notag\\
&=& \log\left\lbrace\frac{\varepsilon_H}{2\pi}\biggl(\sum_{k=1}^\infty\sigma_k^2\biggl)\right\rbrace\int_{x_0}^\infty\frac{\d x}{x^2}
+ (1-2H_{\sup})\int_{x_0}^\infty\frac{\log x}{x^2}\, \d x\notag
- \int_{x_0}^\infty\frac{\log (\lambda^2+x^2)}{x^2}\, \d x \\
&>& -\infty.\label{eq 46}
\end{eqnarray*}
The claim follows now from Theorem 2.1 of \cite{gasbarra2011conditional}.
\end{proof}

\section{Sample Paths}\label{sect:SP}
We illustrate the mmfBm and the mmfOU by simulating their sample paths.

Each of the sample paths, both for the mmfBm and for the mmfOU is given on $N=1000$ equidistant points $t_k=k/(N-1)$ of the time-interval $[0,1]$, with $n=10$ equidistant Hurst exponents $H_i=H_{\inf}+(i-1)(H_{\sup} -H_{\inf} )/(n-1)$ on the Hurst interval $[H_{\inf},H_{\sup}]$. Also, the volatility coefficients $\sigma_i=i^{-1},i!^{-1},\e^{-i}$ are used in the sample paths. In all mmfOU paths $\lambda=1$.

\begin{figure}[H]
\centering
\includegraphics[width=0.98\textwidth
]{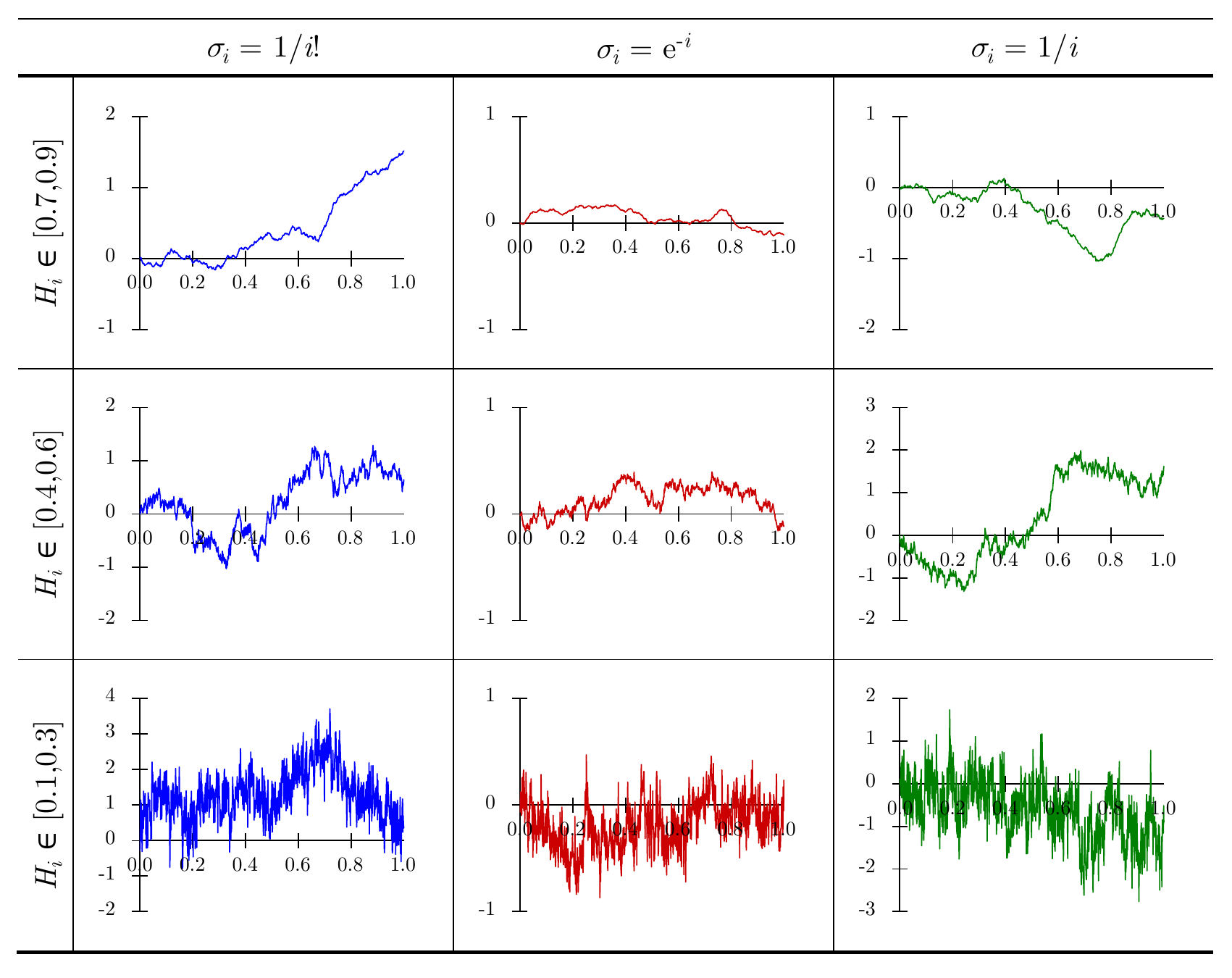}
\caption{Sample paths of mmfBm with equidistant time points and equidistant Hurst parameters.}\label{Fig: mmfBm}
\end{figure}

\begin{figure}[H]
\centering
\includegraphics[width=0.98\textwidth
]{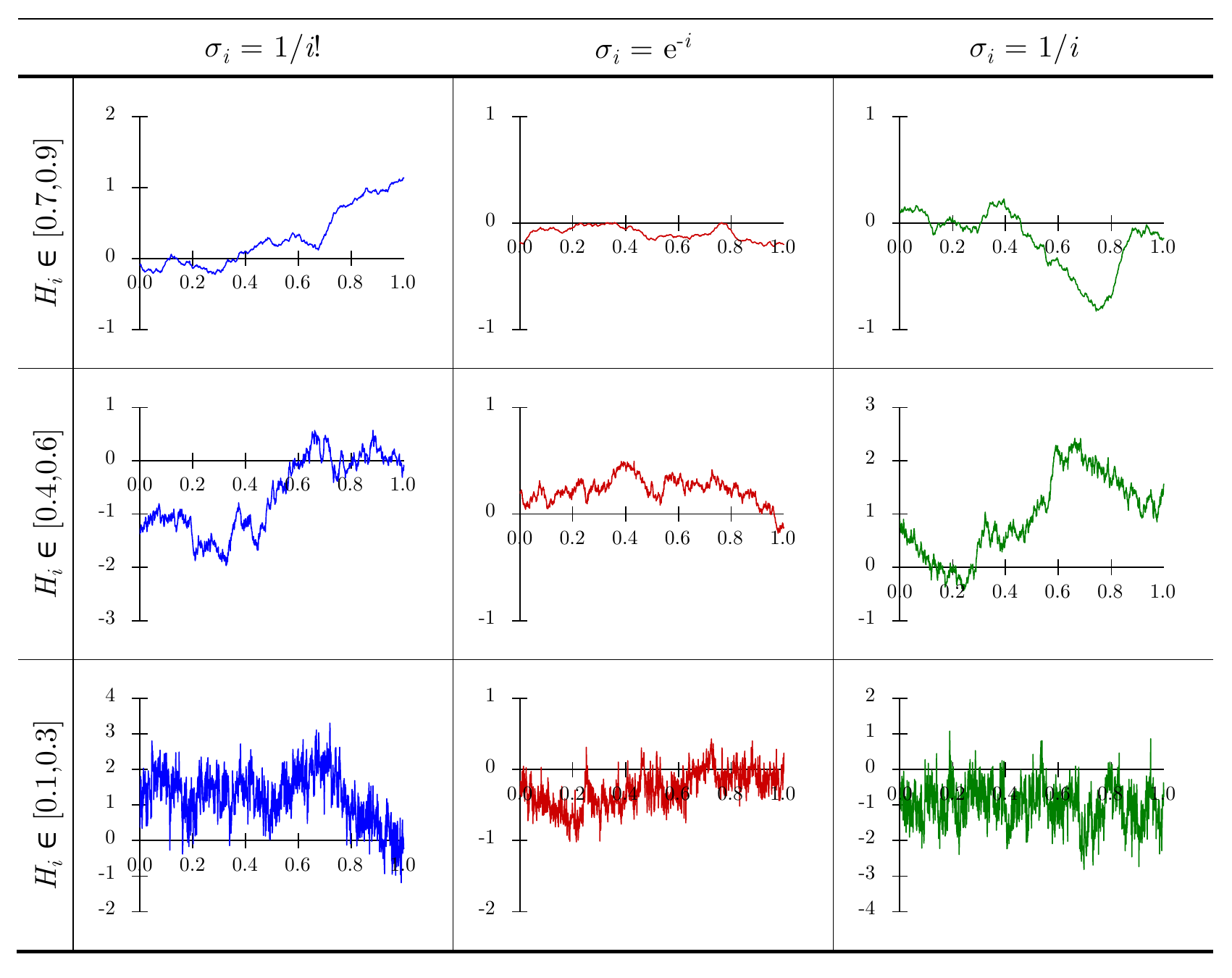}
\caption{Sample paths of mmfOU with equidistant time points and equidistant Hurst parameters.}\label{Fig: mmfOU}
\end{figure}


\bibliographystyle{siam}
\bibliography{pipliateekki.bib}

\end{document}